\documentclass[]{interact}

\usepackage{epstopdf}
\usepackage[caption=false]{subfig}

\usepackage[numbers,sort&compress]{natbib}
\bibpunct[, ]{[}{]}{,}{n}{,}{,}
\makeatletter
\def\NAT@def@citea{\def\@citea{\NAT@separator}}
\makeatother

\theoremstyle{plain}
\newtheorem{theorem}{Theorem}[section]
\newtheorem{lemma}[theorem]{Lemma}

\newtheorem{proposition}[theorem]{Proposition}

\theoremstyle{definition}
\newtheorem{definition}[theorem]{Definition}

\theoremstyle{remark}
\newtheorem{remark}{Remark}

\begin{document}


\title{Nuclear Norm under Tensor Kronecker Products}

\author{
\name{Robert Cochrane \thanks{CONTACT Robert Cochrane. Email: robcoch@umich.edu}}
\affil{Department of Mathematics, University of Michigan, Ann Arbor, USA}
}

\maketitle

\begin{abstract}
Derksen proved that the spectral norm is multiplicative with respect to vertical tensor products (also known as tensor Kronecker products). We will use this result to show that the nuclear norm and other norms of interest are also multiplicative with respect to vertical tensor products.
\end{abstract}


\section{Introduction}

Much work in algebraic complexity theory concerns the \emph{vertical tensor product}, alternatively known as the \emph{tensor Kronecker product}. It is natural to ask which properties of tensors are preserved under this product. It has been known since the work of Strassen that properties such as the tensor rank and border rank are not preserved when taking the vertical tensor product (or even the usual tensor product) of tensors of order at least 3 (\cite{C}, \cite{C2}). Meanwhile, it is clear that the  Frobenius norm of tensors is preserved, and it has been shown that the spectral norm of tensors is also multiplicative under the vertical tensor product (\cite{D}), over both $\mathbb{R}$ and $\mathbb{C}$.

We use this information to show that some other norms on tensor product spaces are also multiplicative with respect to the vertical tensor product - in particular, the nuclear norm has this property. This serves as motivation to investigate how other classical questions involving tensor rank might be answered when rank is replaced by the nuclear norm. We discuss how an analogue of Strassen's direct sum conjecture holds for the nuclear norm.

\section{Notation}

We will follow the notation of Derksen (\cite{D}). In particular, a \emph{d-th order tensor product space} is a pair $\textbf{U}=(U,(U^{(1)}, \ldots , U^{(d)}))$ where each $U^{(i)}$ is a finite dimensional vector space and 
\[U=U^{(1)}\otimes \cdots \otimes U^{(d)}.\]
We will assume all of our vector spaces are over $\mathbb{R}$ or $\mathbb{C}$, and the results hold over both fields, unless otherwise stated.

If $\textbf{U}=(U,(U^{(1)}, \ldots , U^{(d)}))$ and $\textbf{V}=(V,(V^{(1)}, \ldots , V^{(d)}))$ are tensor product spaces, then we define their \emph{vertical tensor product} as 
\[\textbf{U} \boxtimes \textbf{V}=\left(U \otimes V,(U^{(1)} \otimes V^{(1)}, \ldots , U^{(d)} \otimes V^{(d)})\right).\]
For $S \in \textbf{U}$, $T \in \textbf{V}$, we define $S \boxtimes T$ to be $S \otimes T$, viewed as an element of $\textbf{U} \boxtimes \textbf{V}$.  

For $S=(S_1,\ldots ,S_r) \in \textbf{U}^r$, $T=(T_1, \ldots, T_s) \in \textbf{V}^s$, we define 
\[S \boxtimes T = (S_i \boxtimes T_j | 1 \leq i \leq r, 1 \leq j \leq s).\]

Similarly, if $\textbf{U}=(U,(U^{(1)}, \ldots , U^{(d)}))$ and $\textbf{V}=(V,(V^{(1)}, \ldots , V^{(d)}))$ are tensor product spaces, then we define their \emph{direct sum} as
\[\textbf{U} \oplus \textbf{V}=\left(U \oplus V,(U^{(1)} \oplus V^{(1)}, \ldots , U^{(d)} \oplus V^{(d)})\right).\]
For $S \in \textbf{U}$, $T \in \textbf{V}$, we define $S \oplus T$ as an element of $\textbf{U} \oplus \textbf{V}$ in the obvious way. 

For a tensor product space $\textbf{U}$, and $S, T \in \textbf{U}$, we will denote the Frobenius inner product of $S$ and $T$ by $\langle S,T \rangle$, and the associated Frobenius norm (sometimes also known as the Euclidean norm) of $S$ by 
\[\|S\|:=\sqrt{|\langle S,S \rangle|}.\]

For a tensor $S$ in a tensor product space $\textbf{U}$, we may write $S$ as a sum of pure (or simple) tensors
\begin{equation}\label{decomp}
    S= \sum\limits_{i=1}^r v_i \text{ where } v_i=v_i^{(1)} \otimes \ldots \otimes v_i^{(d)} \text{ and } v_i^{(e)} \in U^{(e)}.\tag{1}
\end{equation}

We define the \emph{nuclear norm} of $S$ to be the infimum of $\sum\limits_{i=1}^r \|v_i\|$ over all such decompositions (\ref{decomp}), and denote it by $\|S\|_{\star}$. It can be shown that this infimum can always be achieved (\cite{FL}), and so we may take the minimum over such decompositions instead. 

The \emph{spectral norm} of $S$, denoted $\| S \|_{\sigma}$, is defined to be the maximum value of $|\langle S,u \rangle|$ where $u$ ranges over all pure tensors of unit length. More generally, as in \cite{D}, for a an $r$-tuple of tensors, $\textbf{S}=(S_1, \ldots S_r)$, and $1 \leq \alpha < \infty$, we will define $[\textbf{S}]_{\alpha}$ as the maximum of
\[\left(\sum_{i=1}^r |\langle S_i, u \rangle |^{\alpha} \right)^{1/\alpha}\]
over all pure tensors $u$ of unit length. 
For $\alpha=\infty$, we will define $[\textbf{S}]_{\alpha}$ as the maximum over all pure tensors $u$ of unit length of $\max\limits_i |\langle S_i,u \rangle |$, or equivalently, as $\max\limits_i \| S_i \|_\sigma$

We say two norms $\|\cdot \|_X$ and $\|\cdot \|_Y$ on a tensor product space $\textbf{U}$ are \emph{dual} if for all $S, S' \in \textbf{U}$, we have $|\langle S,S' \rangle| \leq \|S\|_{X}\| S' \|_Y$, and for every $S \in \textbf{U}$, there exists some nonzero $S' \in \textbf{U}$ such that the above inequality becomes an equality. (The last condition is equivalent to: for every $S'\in \textbf{U}$ there exists a nonzero $S\in \textbf{U}$ for which the inequality becomes an  equality.)

\section{Norms under Tensor Kronecker Products}

In this section, $\textbf{U}=(U,(U^{(1)}, \ldots , U^{(d)}))$ and $\textbf{V}=(V,(V^{(1)}, \ldots , V^{(d)}))$ will denote $d$-th order tensor product spaces (over $\mathbb{R}$ or $\mathbb{C}$). Many of the norms we will be interested in can be defined in similar ways on different tensor spaces, and so in an abuse of notation, we will often use $\|\cdot \|_X$ to simultaneously denote the norms $\|\cdot \|_{X,\textbf{U}}$, $\|\cdot \|_{X,\textbf{V}}$, and $\|\cdot \|_{X,\textbf{U} \boxtimes \textbf{V}}$ defined on the spaces $\textbf{U}$, $\textbf{V}$, and $\textbf{U} \boxtimes \textbf{V}$ respectively.

The following is clear by direct calculation:

\begin{proposition} \label{Easylem} If $S$ and $T$ are tensors, then $\|S \otimes T\|=\|S\|\|T\|$. Moreover, if $S \in \textbf{U}$ and $T \in \textbf{V}$ are $d$-th order tensors, then $\| S \boxtimes T\| =\| S\| \| T \|$.
\end{proposition}

We will also make use of the following propositions. The first is well known (e.g. see \cite{LC}). 

\begin{proposition} \label{Derksenlem1} On any tensor product space, $ \|\cdot\|_{\star}$ and $\| \cdot \|_{\sigma}$ are dual.
\end{proposition}

The second appears as Proposition 3.3 in \cite{D}. Though the proof there is performed over $\mathbb{C}$, it also works over $\mathbb{R}$.

\begin{proposition} \label{Derksenlem2} If $\textbf{U}=(U,(U^{(1)}, \ldots , U^{(d)}))$ and $\textbf{V}=(V,(V^{(1)}, \ldots , V^{(d)}))$ are tensor product spaces (over $\mathbb{R}$ or $\mathbb{C}$) with $S \in \textbf{U}^r$ and $T \in \textbf{V}^s$, then we have
\[[ S \boxtimes T]_{\alpha} =[ S]_{\alpha} [ T ]_{\alpha}.\]

In particular, taking $r=s=1$, we see that for $S \in \textbf{U}$, and $T \in \textbf{V}$, we have $\| S \boxtimes T\|_{\sigma} =\| S\|_{\sigma} \| T \|_{\sigma}$.
\end{proposition}

We aim to use the duality of the spectral and nuclear norms, together with the above result on the spectral norm, to make a statement about the nuclear norm of the vertical tensor product of tensors. We require the following Lemma: 

\begin{lemma} \label{Holder} Let $\|\cdot \|_X$ and $\|\cdot \|_Y$ be dual on tensor product spaces $\textbf{U}$, $\textbf{V}$, and $\textbf{U} \boxtimes \textbf{V}$, and suppose that for all $S \in \textbf{U}$ and $T \in \textbf{V}$, we have $\| S \boxtimes T\|_Y \leq \| S\|_Y \| T \|_Y$. Then for all $S \in \textbf{U}$, $T \in \textbf{V}$, we have $\| S \boxtimes T\|_X \geq \| S\|_X \| T \|_X$.  
\end{lemma}

\begin{proof} Pick nonzero $S' \in \textbf{U}$ and $T' \in \textbf{V}$ such that $|\langle S,S' \rangle| = \|S\|_{X}\| S' \|_{Y}$ and $|\langle T,T' \rangle| = \|T\|_{X}\| T' \|_{Y}$. Then,

\begin{align*}
   \|S \boxtimes T\|_{X} \| S' \|_Y \| T' \|_Y  & \geq \|S \boxtimes T\|_{X} \| S' \boxtimes T' \|_Y \\
     & \geq |\langle S \boxtimes T,S' \boxtimes T' \rangle|\\
     & = |\langle S,S' \rangle| |\langle T,T' \rangle|\\
     & = \|S\|_{X} \|T\|_{X} \| S' \|_Y \| T' \|_Y .\\
\end{align*}

Since $S'$ and $T'$ are nonzero, we conclude that $\|S \boxtimes T\|_{X} \geq \|S\|_{X}\|T\|_{X}$.
\end{proof}

Applying the above lemma to Proposition \ref{Derksenlem2} leads to an analogous result for the nuclear norm:

\begin{proposition} If $S \in \textbf{U}$, and $T \in \textbf{V}$ are $d$-th order tensors, then $\|S \boxtimes T\|_{\star}=\|S\|_{\star}\|T\|_{\star}$.
\end{proposition}

\begin{proof} Combining Propositions \ref{Derksenlem1} and  \ref{Derksenlem2} with Lemma \ref{Holder}, we see that $\|S \boxtimes T\|_{\star} \geq \|S\|_{\star}\|T\|_{\star}$.
So it remains to show that $\|S \boxtimes T\|_{\star} \leq \|S\|_{\star}\|T\|_{\star}$.

Let
\[S=\sum_{i=1}^{r_S} u_i \text{ with } u_i= u_i^{(1)} \otimes \ldots \otimes u_i^{(d)} \text{ and } u_i^{(e)} \in U^{(e)}, \]
and similarly,
\[T=\sum_{j=1}^{r_T} v_j \text{ with } v_j= v_j^{(1)} \otimes \ldots \otimes v_j^{(d)} \text{ and } v_j^{(e)} \in V^{(e)}. \]

Then
\[S \boxtimes T= \sum_{i=1}^{r_S} \sum_{j=1}^{r_T} (u_i^{(1)} \otimes v_j^{(1)}) \otimes \ldots \otimes (u_i^{(d)} \otimes v_j^{(d)})\]

and so, applying Proposition \ref{Easylem}, 
\begin{align*} \|S \boxtimes T\|_{\star} &\leq \sum\limits_{i=1}^{r_S} \sum\limits_{j=1}^{r_T} \left\|(u_i^{(1)} \otimes v_j^{(1)}) \otimes \ldots \otimes (u_i^{(d)} \otimes v_j^{(d)})\right\|\\
& = \sum\limits_{i=1}^{r_S} \sum\limits_{j=1}^{r_T} \left\|u_i^{(1)} \otimes \ldots \otimes u_i^{(d)} \right\| \left\| v_j^{(1)} \otimes \ldots \otimes v_j^{(d)}\right\|\\
& = \sum\limits_{i=1}^{r_S} \left\|u_i^{(1)} \otimes \ldots \otimes u_i^{(d)} \right\| \sum\limits_{j=1}^{r_T} \left\| v_j^{(1)} \otimes \ldots \otimes v_j^{(d)}\right\|.\\
\end{align*}

Taking the minima of $\sum\limits_{i=1}^{r_S} \|u_i\|$ and $\sum\limits_{j=1}^{r_S} \|v_j\|$ over all decompositions $S=\sum\limits_{i=1}^{r_S} u_i$ and $T=\sum\limits_{j=1}^{r_T} v_j$, we see that
$\|S \boxtimes T\|_{\star} \leq \|S\|_{\star}\|T\|_{\star}$.
\end{proof}

In fact, we may further generalize the above result by considering tuples of tensors. 

\begin{definition}\label{def:dualnorm}
For $\textbf{S} \in \textbf{U}^r$ an $r$-tuple of tensors with $\textbf{S}=(S_1, \ldots, S_r)$ and $1 \leq \beta < \infty$ we define $[\textbf{S}]_{\beta}^{\star}$ as the minimum of
\begin{equation}\label{eq:betanorm}
\sum_{j=1}^m \left( \sum_{i=1}^r | \lambda_{i,j}|^{\beta} \right)^{1/\beta}
\end{equation}
over all $m$ and all $\{\lambda_{i,j}\}$ for which there exist unit simple tensors $v_1,v_2,\dots,v_m$
and decompositions $S_i=\sum_{j=1}^m \lambda_{i,j}v_j$, $i=1,2,\dots,r$.
For $\beta=\infty$, we define $[\textbf{S}]_{\infty}^{\star}$ by
replacing (\ref{eq:betanorm}) by
\[\sum_{j=1}^m\max_{1\leq i\leq r} | \lambda_{i,j}|.\]
\end{definition}

\begin{remark} The  minimum is the previous definition is well-defined. To see this, consider the compact set
\[\left\lbrace(k_1 v, k_2 v, \ldots k_r v) \middle| \sum_{i=1}^r |k_i|^{\beta}=1, v \text{ is a unit simple tensor}  \right\rbrace.\]
and let $B$ be its convex hull.
 For an $r$-tuple $\textbf{S}$ of tensors, it is easy to see that $[\textbf{S}]_{\beta}^{\star}$ is the infimum of all $t$ such that $\textbf{S} \in tB$. But if $\textbf{S} \in tB$, then by Carath\'eodory's Convexity Theorem (see~\cite[Theorem 2.3]{B}), we can find decompositions $S_i=\sum\limits_{j=1}^m \lambda_{i,j} v_{j}$ with $\sum\limits_{j=1}^m \left( \sum\limits_{i=1}^r | \lambda_{i,j}|^{\beta} \right)^{1/\beta} \leq t$, where $m\leq \dim \textbf{U}^r+1$. So in Definition~\ref{def:dualnorm} we may take $m=\dim_{\mathbb R}\textbf{U}^r+1$. The set of all $\lambda_{i,j}$ and $v_j$ for which $S_i=\sum\limits_{j=1}^m \lambda_{i,j} v_{j}$
 ($1\leq i\leq r$) is closed, so the function~(\ref{eq:betanorm}) has a minimum on this set.
\end{remark}

We make the following key observation:

\begin{proposition} If $1 \leq \alpha, \beta \leq \infty$ are H\"older conjugates (i.e $\frac{1}{\alpha} + \frac{1}{\beta}=1$ or $\lbrace\alpha, \beta\rbrace=\lbrace1,\infty\rbrace$), then $[\cdot]_\alpha$ and $[\cdot]_{\beta}^{\star}$ are dual.
\end{proposition}
\begin{proof}
Let $\textbf{S}=(S_1, \ldots S_r)$ and $\textbf{T}=(T_1, \ldots T_r)$ be $r$-tuples of tensors, and write $T_i=\sum\limits_j \lambda_{i,j} v_{j}$ with $v_{j}$ simple unit tensors. Then, for $1 < \alpha, \beta < \infty$, using the H\"older inequality, 
\begin{align*} \left|\left\langle \textbf{S}, \textbf{T} \right\rangle\right| & = \left|\sum_i \left\langle S_i, \sum_j \lambda_{i,j}v_j \right\rangle\right| \\
& \leq \sum_i \sum_j |\lambda_{i,j}| |\langle S_i, v_j \rangle| \\
& \leq \sum_j \left(\sum_i | \langle S_i, v_j \rangle|^{\alpha} \right)^{1/\alpha} \left(\sum_i |\lambda_{i,j}|^{\beta}\right)^{1/\beta}\\
& \leq [\textbf{S}]_{\alpha} \sum_j \left( \sum_i | \lambda_{i,j}|^{\beta} \right)^{1/\beta}.\\
\end{align*}
Taking the minimum over all decompositions $T_i=\sum\limits_j \lambda_{i,j} v_{j}$ with $v_{j}$ simple unit tensors gives $\left|\left\langle \textbf{S}, \textbf{T} \right\rangle\right| \leq [\textbf{S}]_{\alpha}[\textbf{T}]_{\beta}^{\star}$. The same inequality holds for $\lbrace\alpha, \beta\rbrace=\lbrace1,\infty\rbrace$ using the same reasoning.

Again suppose that $1 < \alpha, \beta < \infty$. Let $u$ be a simple unit tensor such that
\[[\textbf{S}]_{\alpha}=\left(\sum_{i=1}^r |\langle S_i, u \rangle |^{\alpha} \right)^{1/\alpha}\]
and take $\textbf{T}=(T_1,\ldots,T_r)$ with $T_i=|\langle S_i, u \rangle |^{\alpha/\beta-1}\overline{\langle S_i, u \rangle} u$, where $\overline{\langle S_i, u \rangle}$ denotes the complex conjugate of $\langle S_i, u \rangle$ (or just denotes $\langle S_i, u \rangle$ if our ground field is $\mathbb{R}$). Then by definition, 
\[[\textbf{T}]_{\beta}^{\star} \leq \left( \sum_i \left(|\langle S_i, u \rangle |^{\alpha/\beta}\right)^{\beta} \right)^{1/\beta} = \left( \sum_i |\langle S_i, u \rangle |^{\alpha} \right)^{1/\beta}.\]

We also have
\[|\langle S_i, T_i \rangle |=|\langle S_i, u \rangle |^{\alpha/\beta+1}=|\langle S_i, u \rangle |^{\alpha/\beta+\alpha/\alpha}=|\langle S_i, u \rangle |^{\alpha}\]
and hence, since $\frac{1}{\alpha} + \frac{1}{\beta}=1$,
\begin{align*}\left|\left\langle \textbf{S}, \textbf{T} \right\rangle\right|&=\sum_i |\langle S_i, u \rangle |^{\alpha}\\
&=\left(\sum_i |\langle S_i, u \rangle |^\alpha\right)^{1/\alpha}\left(\sum_i|\langle S_i, u \rangle |^{\alpha}\right)^{1/\beta}\\
&\geq [\textbf{S}]_{\alpha}[\textbf{T}]_{\beta}^{\star}. 
\end{align*}

Thus, given any $r$-tuple of tensors $\textbf{S}$ we can construct a nonzero $\textbf{T}$ such that $|\langle \textbf{S}, \textbf{T} \rangle|=[\textbf{S}]_{\alpha}[\textbf{T}]_{\beta}^{\star}$, and so the norms are dual, as claimed.  

The proof for $\alpha=1$ and $\beta=\infty$ is similar - in particular, if $u$ is a unit simple tensor such that $[\textbf{S}]_{1}=\sum_{i=1}^r |\langle S_i, u \rangle |$ and we take $T_i=u$ for all $i$, then $|\langle \textbf{S}, \textbf{T} \rangle|=[\textbf{S}]_{1}=[\textbf{S}]_{1}[\textbf{T}]_{\infty}^{\star}$.

Similarly, for $\alpha=\infty$ and $\beta=1$, if $u$ is a unit simple tensor such that $[\textbf{S}]_{\infty}=\max\limits_{i} |\langle S_i, u \rangle |=|\langle S_k, u \rangle |$ for some $k$, then let $T_k=|\langle S_k, u \rangle |u$ and $T_i=0$ otherwise. Then $|\langle \textbf{S}, \textbf{T}\rangle|=|\langle S_k, u \rangle |^2=[\textbf{S}]_{\infty}[\textbf{T}]_{1}^{\star}$.
\end{proof}

\begin{proposition} If $\textbf{U}=(U,(U^{(1)}, \ldots , U^{(d)}))$ and $\textbf{V}=(V,(V^{(1)}, \ldots , V^{(d)}))$ are tensor product spaces (over $\mathbb{R}$ or $\mathbb{C}$) with $S \in \textbf{U}^r$, $T \in \textbf{V}^s$, and $1 \leq \beta \leq \infty$, then $[S \boxtimes T]_{\beta}^{\star}=[S]_{\beta}^{\star}[T]_{\beta}^{\star}$.
\end{proposition}

\begin{proof} Again combining Lemma \ref{Holder} and Proposition \ref{Derksenlem2}, we see that $[S \boxtimes T]_{\beta}^{\star} \geq [S]_{\beta}^{\star}[T]_{\beta}^{\star}$.

Suppose that $1 \leq \beta < \infty$. Let $u_1, \ldots, u_n$ be unit simple tensors such that $S_i=\sum\limits_{j=1}^n \lambda_{i,j} u_j$ and 
\[[S]_\beta^\star=\sum_{j=1}^n \left( \sum_{i=1}^r | \lambda_{i,j}|^{\beta} \right)^{1/\beta}\]
and similarly, let $v_1, \ldots, v_m$ be unit simple tensors such that $T_k=\sum\limits_{l=1}^m \mu_{k,l} v_l$ and 
\[[T]_\beta^\star=\sum_{l=1}^m \left( \sum_{k=1}^s | \mu_{k,l}|^{\beta} \right)^{1/\beta}.\]

Then $S_i \boxtimes T_k=\sum\limits_{j=1}^n \sum\limits_{l=1}^m \lambda_{i,j} \mu_{k,l} u_j \boxtimes v_l$ and so

\begin{align*}
[S \boxtimes T]_\beta^\star & \leq \sum_{j=1}^n \sum_{l=1}^m \left( \sum_{i=1}^r \sum_{k=1}^s |\lambda_{i,j}|^\beta |\mu_{k,l} |^\beta \right)^{1/\beta} \\
&= \sum_{j=1}^n \sum_{l=1}^m \left(\left( \sum_{i=1}^r |\lambda_{i,j}|^\beta\right)^{1/\beta} \left( \sum_{k=1}^s |\mu_{k,l} |^\beta \right)^{1/\beta}\right) \\
&=\sum_{j=1}^n  \left( \sum_{i=1}^r |\lambda_{i,j}|^\beta \right)^{1/\beta} \sum_{l=1}^m \left( \sum_{k=1}^s |\mu_{k,l} |^\beta \right)^{1/\beta}\\
&=[S]_\beta^\star [T]_\beta^\star.\\
\end{align*}
By the same method, we can show that $[S \boxtimes T]_{\infty}^{\star} \leq [S]_{\infty}^{\star}[T]_{\infty}^{\star}$

Hence for all $1 \leq \beta \leq \infty$, we have $[S \boxtimes T]_{\beta}^{\star} = [S]_{\beta}^{\star}[T]_{\beta}^{\star}$.
\end{proof}

\begin{remark} A classical question about the rank of tensors was posed by Strassen (\cite{S2}). He posited that the rank of tensors is additive with respect to direct sums. Although his conjecture was later shown to be false (\cite{S1}), it was recently shown (\cite{KLW}) by Kong, Li, and Wang that the nuclear norm is additive with respect to direct sums, i.e., if $\textbf{U}=(U,(U^{(1)}, \ldots , U^{(d)}))$ and $\textbf{V}=(V,(V^{(1)}, \ldots , V^{(d)}))$ are tensor product spaces with $d>1$, and $S \in \textbf{U}$ and $T \in \textbf{V}$, then $\|S \oplus T\|_{\star}=\|S\|_{\star}+\|T\|_{\star}$. 

Similarly, the authors showed that if $\textbf{U}=(U,(U^{(1)}, \ldots , U^{(d)}))$ and $\textbf{V}=(V,(V^{(1)}, \ldots , V^{(d)}))$ are tensor product spaces with $d>1$, and $A \in \textbf{U}$ and $B \in \textbf{V}$, then $\|A \oplus B\|_{\sigma}=\max\lbrace \|A\|_\sigma,\|B\|_\sigma\rbrace$. 

Although the authors state their results only over $\mathbb{R}$, they also hold over $\mathbb{C}$.
\end{remark}

\section*{Acknowledgements} The author is grateful to his advisor, Harm Derksen, for all of his help and comments. The author was partially supported by NSF grant 1837985.

\end{document}